\documentclass[reqno]{amsart}
\usepackage[hidelinks]{hyperref}
\usepackage{amssymb}
\usepackage{enumerate}

\begin{document}

\title[Rayleigh-Stokes equations with delays...]
{On stability for semilinear generalized Rayleigh-Stokes equation involving delays}

\author[T.D. Ke, D.Lan, P.T. Tuan]{Tran Dinh Ke, Do Lan *, Pham Thanh Tuan}

\address{Tran Dinh Ke  \hfill\break
Department of Mathematics, Hanoi National University of Education \hfill\break
136 Xuan Thuy, Cau Giay, Hanoi, Vietnam}
\email{ketd@hnue.edu.vn} 

\address{Do Lan\hfill\break
Faculty of Computer Engineering and Science, Thuyloi University \hfill\break
175 Tay Son, Dong Da, Hanoi, Vietnam}
\email{dolan@tlu.edu.vn (D. Lan)}

\address{Pham Thanh Tuan  \hfill\break
Department of Mathematics, Hanoi Pedagogical University 2,\hfill\break
Xuan Hoa, Phuc Yen, Vinh Phuc, Vietnam}
\email{phamthanhtuan@hpu2.edu.vn}

\subjclass[2010]{35B40,35R11,35C15,45D05,45K05}
\keywords{Rayleigh-Stokes problem; stability; nonlocal PDE}
\thanks{* Corresponding author. Email: dolan@tlu.edu.vn (D. Lan)}
\maketitle
\numberwithin{equation}{section}
\newtheorem{theorem}{Theorem}[section]
\newtheorem{lemma}[theorem]{Lemma}
\newtheorem{proposition}[theorem]{Proposition}
\newtheorem{corollary}[theorem]{Corollary}
\newtheorem{definition}{Definition}[section]
\newtheorem{remark}{Remark}[section]

\begin{abstract}
We consider a functional semilinear Rayleigh-Stokes equation involving fractional derivative. Our aim is to analyze some circumstances, in those the global solvability, and asymptotic behavior of solutions are addressed. By establishing a Halanay type inequality, we show the dissipativity and asymptotic stability of solutions to our problem. In addition, we prove the existence of a compact set of decay solutions by using local estimates and fixed point arguments.
\end{abstract}

\section{Introduction}
Let $\Omega\subset \mathbb R^d$ be a bounded domain with smooth boundary $\partial\Omega$. Consider the following problem
\begin{align}
\partial_t u - (1+\gamma\partial^\alpha_t)\Delta u & = f(t, u_\rho)\;\text{ in }\Omega, t>0,\label{e1}\\
u & = 0\; \text{ on } \partial\Omega,\; t\ge 0,\label{e2}\\
u(x, s) & = \xi(x,s), \; x\in \Omega, s\in [-\tau,0],\label{e3}
\end{align}
where $\gamma>0$, $\alpha\in (0,1)$, $\partial_t=\frac{\partial}{\partial t}$, $\partial_t^\alpha$ stands for the Riemann-Liouville derivative of order $\alpha$ defined by
$$
\partial_t^\alpha v(t) = \frac{d}{dt}\int_0^t g_{1-\alpha}(t-s)v(s)ds,
$$
where $g_\beta(t) =\dfrac{t^{\beta-1}}{\Gamma(\beta)}$ for $\beta>0, t>0$. In this model, $u_\rho$ is defined by $u_\rho(x,t)=u(x,t-\rho(t))$ with $\rho$ being a continuous function on $\mathbb R^+$ such that $-\tau\le t-\rho(t)\le t$ and $\lim\limits_{t\to\infty}(t-\rho(t))=\infty$, $f:\mathbb R^+\times L^2(\Omega)\to L^2(\Omega)$ is a nonlinear map and $\xi\in C([-\tau,0];L^2(\Omega))$ is given.

Equation \eqref{e1} arose in a generalized Rayleigh-Stokes problem, where its constitution was given in \cite{FJFV09, STZM06}. This type of equation is employed to describe the behavior of flow of non-Newtonian fluids occupied in cylinders. In this model, the term of fractional derivative gives a significant description for the viscoelasticity of fluids under examination.

As a matter of fact, various numerical methods have been developed for solving Rayleigh-Stokes problem in linear case, see e.g. \cite{Bazh15,Bi18,Chen13,Chen08,SSM18,Zaky18}. We also mention the analytic representation for solution of  this problem in  \cite{FJFV09,Khan09,STZM06,XN09,ZBF07}. Recently, some inverse problems involving \eqref{e1} has been addressed in \cite{Luc19,Ngoc21,Tuan19}, where the state function is identified from its terminal value. In the case of non-delayed, i.e. $f=f(u)$, the regularity and stability of solution to Rayleigh-Stokes equation has been analyzed in the recent work \cite{Lan21}.

In the present work, we concern with the nonlinear model, where the nonlinearity $f$ contains a delayed term, in order to describe the situation that the external force depends on history state of the system. It is worth noting that, the appearance of delayed term may reduce the performance and routinely affect the stability of the system. A typical example for the delayed term is that, $\rho(t)=(1-q)t+\tau$, $u_\rho(x,t)=u(x,qt-\tau)$, for some $q\in (0,1]$, which is a proportional delay. For this model, the long-time behavior of solutions is an issue that has not been addressed in literature, and we aim at closing this gap. We first prove that, the problem is globally solvable in both cases when $f$ has a linear or superlinear growth. Then we analyze some sufficient conditions ensuring the dissipativity and asymptotic stability for our system. Finally, we show the existence of a compact set of decay solutions in the case $f$ is of superlinear.

Our work is as follows. In the next section, we recall typical properties of relaxation function and prove a Halanay type equality for using in stability analysis. Section 3 is devoted to proving the solvability and stability results. In the last section, we present the existence of a compact set of decay solutions to our problem.

\section{Preliminaries}
We first give a representation of solutions to \eqref{e1}-\eqref{e3} by using a resolvent operator. Consider the relaxation problem
\begin{align}
\omega'(t) + \mu (1+\gamma\partial_t^\alpha) \omega (t) & = 0,\; t> 0,\label{re1}\\
\omega (0) & = 1,\label{re2}
\end{align}
where the unknown $\omega$ is a scalar function, $\mu$ and $\gamma$ are positive parameters. We collect some properties of $\omega$ in the following proposition.
\begin{proposition}\cite{Lan21}\label{pp-relax-func}
Let $\omega$ be the solution of \eqref{re1}-\eqref{re2}. Then
\begin{enumerate}
\item $0<\omega (t)\le 1$ for all $t\ge 0$.
\item The function $\omega$ is completely monotone for $t\ge 0$, i.e. $(-1)^n \omega^{(n)}(t)\ge 0$ for $t\ge 0$ and $n\in\mathbb N$. Consequently, $\omega$ is a nonincreasing function.
\item $\mu \omega(t) \le (t+g_{2-\alpha}(t))^{-1} \le \min\{t^{-1}, t^{\alpha-1}\},$ for all $t>0$.
\item $\displaystyle \int_0^t \omega(s)ds\le \mu^{-1}(1-\omega(t))$, for any $t>0$.
\item For fixed $t\ge 0$ and $\gamma>0$, the function $\mu\mapsto \omega(t, \mu)$ is nonincreasing on $[0,\infty)$.
\end{enumerate}
\end{proposition}
Denote by $\omega(\cdot, \mu)$ the solution of \eqref{re1}-\eqref{re2}, respecting to parameter $\mu$. In what follows, we use the notation $u*v$ to express the Laplace convolution of $u$ and $v$, i.e.
$$
(u*v)(t) = \int_0^t u(t-s)v(s)ds,\; u,v\in L^1_{loc}(\mathbb R^+).
$$
We now concern with the inhomogeneous problem
\begin{align}
v'(t) + \mu(1+\gamma\partial_t^\alpha)v(t) & = g(t), t>0,\label{re3}\\
v(0) & = v_0,\label{re4}
\end{align}
where $\mu>0$, $\gamma>0$ and $g$ is a continuous function. The representation of $v$ is given in the next proposition.
\begin{proposition}\cite{Lan21}\label{pp-relax-func3}
The solution of \eqref{re3}-\eqref{re4} is given by
$$
v(t) = \omega(t,\mu)v_0 + \omega(\cdot,\mu)*g(t), \; t\ge 0,
$$
where $\omega$ is defined by \eqref{re1}-\eqref{re2}.
\end{proposition}

Let $\{\varphi_n\}_{n=1}^\infty$ be the orthonormal basis of $L^2(\Omega)$ consisting of the eigenfunctions of the Laplacian $-\Delta$ subject to homogeneous Dirichlet boundary condition, that is
$$
-\Delta\varphi_n = \lambda_n\varphi_n\text{ in }\Omega, \; \varphi_n = 0\text{ on }\partial\Omega,
$$
where we can assume that $\{\lambda_n\}_{n=1}^\infty$ is an increasing sequence, $\lambda_n>0$ and $\lambda_n\to +\infty$ as $n\to\infty$. Then one can give a representation of solution to the linear problem
\begin{align}
\partial_t u  - (1+\gamma\partial_t^\alpha)\Delta u & = F \; \text{ in }\Omega, t> 0,\label{le1}\\
u & = 0 \; \text{ on } \partial\Omega, t\ge 0,\label{le2}\\
u(\cdot,0) & = \xi\;\text{ in } \Omega,\label{le3}
\end{align}
where $F\in L^1_{loc}(\mathbb R^+; L^2(\Omega))$ and $\xi\in L^2(\Omega)$. 
Indeed, let 
\begin{align*}
u(x,t) & = \sum_{n=1}^\infty u_n(t)\varphi_n(x),\\
F(x,t) & = \sum_{n=1}^\infty F_n(t)\varphi_n(x), \;\xi(x) = \sum_{n=1}^\infty \xi_n\varphi_n(x).
\end{align*}
Then
$$
u_n'(t) + \lambda_n(1+\gamma\partial_t^\alpha)u_n(t) = F_n(t),\; u_n(0) = \xi_n.
$$
Employing Proposition \ref{pp-relax-func3}, we get
$$
u_n(t) = \omega(t,\lambda_n)\xi_n + \int_0^t \omega(t-s,\lambda_n)F_n(s)ds.
$$
This implies
\begin{equation}\label{eq-sol}
u(\cdot,t) = S(t)\xi + \int_0^t S(t-s)F(\cdot,s)ds,
\end{equation}
where $S(t): L^2(\Omega)\to L^2(\Omega)$ is the resolvent operator defined by
\begin{align}\label{op-sol}
S(t)\xi =  \sum_{n=1}^\infty  \omega(t,\lambda_n)\xi_n\varphi_n.
\end{align}
We recall some properties of the resolvent operator in the following lemma.
\begin{lemma}\cite{Bazh15}\label{lm-op-sol}
For any $v\in L^2(\Omega)$, $T>0$, we have:
\begin{enumerate}
\item $S(\cdot)v\in C([0,T];L^2(\Omega))\cap C((0,T];H^2(\Omega)\cap H_0^1(\Omega))$.
\item $\|S(t)v\|\le \omega(t,\lambda_1)\|v\|$, for all $t\ge 0$. In particular, $\|S(t)\|\le 1$ for all $t\ge 0$.
\item $S(\cdot)v\in C^{(m)}((0,T];L^2(\Omega))$ for all $m\in \mathbb N$, and $\|S^{(m)}(t)v\| \le Ct^{-m}\|v\|$, where $C$ is a positive constant.
\item $\|\Delta S^{(m)}(t)v\|\le Ct^{-m-1+\alpha}\|v\|$ for all $t>0$ and $m\in \mathbb N$.
\end{enumerate}
\end{lemma}

Consider the \textit{Cauchy operator} $\mathcal Q: C([0,T];L^2(\Omega))\to C([0,T];L^2(\Omega))$ given by
\begin{equation}\label{cauchy-op}
\mathcal Q(g)(t) = \int_0^t S(t-s)g(s)ds.
\end{equation}
Denote by $\|\cdot\|_\infty$ the sup norm in $C([0,T];L^2(\Omega))$, i.e. $\|g\|_\infty = \sup\limits_{t\in [0,T]}\|g(t)\|$. The following lemma shows the compactness of $\mathcal Q$.
\begin{lemma}\cite{Lan21}\label{lm-cauchy-op}
The Cauchy operator defined by \eqref{cauchy-op} is compact.
\end{lemma}

We are in a position to prove a Halanay type inequality for the stability analysis in the next section.
\begin{lemma}\label{lm-halanay}
Let $v$ be a continuous and nonnegative function satisfying
\begin{align}
v(t) & \le \omega(t,\mu)v_0 + \int_0^t \omega(t-s,\mu)[a \sup_{\zeta\in [s-\rho(s),s]}v(\zeta) + b(s)]ds, \; t>0,\label{halanay1}\\
v(s) & = \psi(s), s\in  [-\tau,0],\label{halanay2}
\end{align}
where $0<a<\mu$, $\psi\in C([-\tau,0];\mathbb R^+)$ and $b\in L^1_{loc}(\mathbb R^+)$ which is nondecreasing. Then
\begin{align}
v(t) & \le \frac{\mu}{\mu-a} \Big[v_0 +  \int_0^t \omega(t-s,\mu)b(s)ds\Big] + \sup_{s\in [-\tau,0]}\psi(s),\;\forall t>0.\label{halanay3}
\end{align}
In addition, if $\omega(\cdot,\mu)*b$ is bounded on $\mathbb R^+$ then 
\begin{align}
\limsup_{t\to \infty} v(t) &\le \sup_{t\in \mathbb R^+}\int_0^t \omega(t-s,\mu)b(s)ds.\label{halanay4}
\end{align}
In particular, if $b=0$ then $v(t)\to 0$ as $t\to \infty$.
\end{lemma}
\begin{proof}
We make use of the following result \cite{Wang15}: if $v\in C([-\tau, \infty);\mathbb R^+)$ is a nonnegative function satisfying
\begin{align*}
v(t) & \le d(t) + c\sup_{\zeta\in [-\tau,t]}v(\zeta), \; t>0,\\
v(s) & = \psi(s), \; s\in [-\tau,0],
\end{align*}
where $d(\cdot)$ is a nondecreasing function and $c\in (0,1)$, then
\begin{align}
v(t) \le (1-c)^{-1}d(t) + \sup_{s\in [-\tau,0]}\psi(s),\; \forall t>0.\label{halanay5}
\end{align}
It follows from \eqref{halanay1} that
\begin{align*}
v(t) & \le v_0 + \omega(\cdot,\mu)*b(t) + a \sup_{\zeta\in [-h,t]}v(\zeta) \int_0^t \omega(t-s,\mu)ds\\
& \le v_0 + \omega(\cdot,\mu)*b(t) +\frac{a}{\mu} \sup_{\zeta\in [-h,t]}v(\zeta) (1-\omega(t,\mu))\\
& \le v_0 + \omega(\cdot,\mu)*b(t) +\frac{a}{\mu} \sup_{\zeta\in [-h,t]}v(\zeta),
\end{align*}
here we utilized Proposition \ref{pp-relax-func}(4). Since $b(\cdot)$ is nondecreasing, it is easily seen that the function $\omega(\cdot,\mu)*b$ is nondecreasing as well. Applying inequality \eqref{halanay5} for $d(\cdot)=\omega(\cdot,\mu)*b$ and $c=a/\mu$, we get \eqref{halanay3} as desired.

Now assume that $\omega(\cdot,\mu)*b$ is bounded on $\mathbb R^+$. Then by \eqref{halanay3}, $v(\cdot)$ is bounded by
$$
M:= \frac{\mu}{\mu-a} \Big[v_0 +  \sup_{t\in\mathbb R^+}\int_0^t \omega(t-s,\mu)b(s)ds\Big] + \sup_{s\in [-\tau,0]}\psi(s),
$$
and therefore the limit $L=\lim\limits_{t\to \infty}\sup_{\zeta\in [t,\infty)}v(\zeta)$ exists. Since $t-\rho(t)\to \infty$ as $t\to \infty$, for any $\varepsilon>0$, one can find $T_1>0$ such that 
$$
\sup_{\zeta\in [t-\rho(t), t]}v(\zeta)\le \sup_{\zeta\in [t-\rho(t), \infty]}v(\zeta)\le L+\varepsilon, \; \forall t\ge T_1.
$$
Owing to the last estimate, we see that
\begin{align}
v(t) & \le \omega(t,\mu)v_0 + \omega(\cdot,\mu)*b(t)\notag\\
& \quad + \left(\int_0^{T_1}+\int_{T_1}^t\right)\omega(t-s,\mu)a\sup_{\zeta\in [s-\rho(s),s]}v(\zeta)ds\notag \\
& \le \omega(t,\mu)v_0 + \omega(\cdot,\mu)*b(t)\notag\\
& \quad + aM\int_0^{T_1}\omega(t-s,\mu) ds+a(L+\varepsilon)\int_{T_1}^t\omega(t-s,\mu)ds\notag \\
& \le \varepsilon v_0 + \omega(\cdot,\mu)*b(t)\notag\\
& \quad + aM\int_{t-T_1}^t \omega(t-s,\mu) ds+a(L+\varepsilon)\int_0^t\omega(t-s,\mu)ds\notag\\
& \le \varepsilon v_0 + \omega(\cdot,\mu)*b(t) + aM\varepsilon + a(L+\varepsilon)\mu^{-1},\label{halanay6}
\end{align}
provided $t$ chosen such that
\begin{align*}
\omega(t,\mu) \le \varepsilon,\; \int_{t-T_1}^t \omega(t-s,\mu) ds \le \varepsilon,
\end{align*}
which is possible since $\omega(t,\mu)\to 0$ as $t\to\infty$ and $\omega(\cdot,\mu)\in L^1(\mathbb R^+)$.

It follows from \eqref{halanay6} that
\begin{align*}
L=\lim_{t\to\infty}\sup_{\zeta\in [t,\infty]}v(\zeta) \le aL\mu^{-1} + \sup_{t\in\mathbb R^+}\omega(\cdot,\mu)*b(t)+(v_0+aM+a\mu^{-1})\varepsilon,
\end{align*}
which implies that
\begin{align*}
L\le \frac{\mu}{\mu-a}\sup_{t\in\mathbb R^+}\omega(\cdot,\mu)*b(t) + \frac{\mu}{\mu-a}(v_0+aM+a\mu^{-1})\varepsilon.
\end{align*}
Hence
\begin{align*}
\limsup_{t\to\infty}v(t) \le L \le \frac{\mu}{\mu-a}\sup_{t\in\mathbb R^+}\omega(\cdot,\mu)*b(t),
\end{align*}
thanks to the fact that $\varepsilon$ is an arbitrarily positive number.
\end{proof}
\section{Solvability and stability}
Based on representation \eqref{eq-sol}, we give the following definition.
\begin{definition}
Let $\xi\in C([-\tau,0];L^2(\Omega))$ be given. A function $u\in C([-\tau,T];L^2(\Omega))$ is said to be a mild solution to \eqref{e1}-\eqref{e3} on the interval $[-\tau,T]$ iff $u(\cdot,s)=\xi(\cdot,s)$ for $s\in [-\tau,0]$ and 
\begin{align*}
u(\cdot,t) = S(t)\xi(\cdot,0) + \int_0^t S(t-s)f(s, u_\rho(\cdot,s))ds, \; t\in [0,T].
\end{align*}
\end{definition}
For given $\xi\in C([-\tau,0];L^2(\Omega))$, denote $C_\xi([0,T];L^2(\Omega)):=\{u\in C([0,T];L^2(\Omega)): u(\cdot,0)=\xi(\cdot,0)\}$. For $u\in C_\xi([0,T];L^2(\Omega))$, we define $u[\xi] \in C([-\tau, T];L^2(\Omega))$ as follows
\begin{equation*}
u[\xi](\cdot,t)=\begin{cases}
u(\cdot, t)& \text{ if } t\in [0,T],\\
\xi(\cdot, t) & \text{ if } t\in [-\tau, 0].
\end{cases}
\end{equation*}
Hence, we have
$$
u[\xi]_\rho(\cdot, t)=\begin{cases}
u(\cdot, t-\rho(t))& \text{ if } t-\rho(t)\in [0,T],\\
\xi(\cdot, t-\rho(t)) & \text{ if } t-\rho(t)\in [-\tau, 0].
\end{cases}
$$
In what follows, we use the notation $\|\cdot\|_\infty$ for the sup norm in the spaces $C([-\tau,0];L^2(\Omega))$, $C([-\tau,T];L^2(\Omega))$ and $C([0,T];L^2(\Omega))$.

Let $\Phi: C_\xi([0,T];L^2(\Omega))\to C_\xi([0,T];L^2(\Omega))$ be the operator defined by
\begin{align*}
\Phi(u)(\cdot, t) = S(t)\xi(\cdot,0) + \int_0^t S(t-s)f(s, u[\xi]_\rho(\cdot,s))ds,
\end{align*}
which will be referred to as \textit{the solution operator}. This operator is continuous if $f$ is a continuous map. Obviously, $u$ is a fixed point of $\Phi$ iff $u[\xi]$ is a mild solution of \eqref{e1}-\eqref{e3}.

In the next theorems, we show some global existence results for \eqref{e1}-\eqref{e3}.
\begin{theorem}\label{th-exist-a}
Let $f: [0,T]\times L^2(\Omega)\to L^2(\Omega)$ be a continuous mapping such that
\begin{itemize}
\item[(\textbf{F}1)] $\|f(t,v)\| \le p(t)G(\|v\|)$ for all $t\in [0,T]$ and $v\in L^2(\Omega)$,
where $p\in L^1(0,T)$ is a nonnegative function and $G$ is a continuous and nonnegative function obeying that
\begin{align*}
\limsup_{r\to 0}\frac{G(r)}r \cdot \sup_{t\in [0,T]}\int_0^t\omega(t-s,\lambda_1)p(s)ds <1.
\end{align*}
\end{itemize}
Then there exists $\delta>0$ such that the problem \eqref{e1}-\eqref{e3} has at least one mild solution on $[-\tau,T]$, provided $\|\xi\|_\infty\le \delta$.
\end{theorem}
\begin{proof}
Let 
\begin{align*}
\ell = \limsup_{r\to 0}\frac{G(r)}r,\; M= \sup_{t\in [0,T]}\omega(\cdot,\lambda_1)*p(t).
\end{align*}
Then by assumption, one can take $\epsilon>0$ such that $(\ell + \epsilon)M<1$. In addition, there is $\eta>0$ such that
\begin{align*}
\frac{G(r)}r \le \ell + \epsilon,\;\forall r\in [0,2\eta].
\end{align*}
Let 
\begin{align*}
\delta_0 = \eta \inf_{t\in [0,T]}\big\{\big[\omega(t,\lambda_1)+(\ell+\epsilon)\omega(\cdot,\lambda_1)*p(t)\big]^{-1}\big[1-(\ell+\epsilon)\omega(\cdot,\lambda_1)*p(t)\big]\big\},
\end{align*}
then $\delta_0>0$. Indeed, observing that 
$$
\big[\omega(t,\lambda_1)+(\ell+\epsilon)\omega(\cdot,\lambda_1)*p(t)\big]^{-1}\ge \left[1+(\ell+\epsilon)M\right]^{-1},
$$ 
we get
\begin{align*}
\delta_0 & \ge \eta \left[1+(\ell+\epsilon)M\right]^{-1}\inf_{t\in [0,T]}\big[1-(\ell+\epsilon)\omega(\cdot,\lambda_1)*p(t)\big]\\
& \ge \eta\left[1+(\ell+\epsilon)M\right]^{-1}\big[1-(\ell+\epsilon)\sup_{t\in [0,T]}\omega(\cdot,\lambda_1)*p(t)\big]\\
& = \eta \left[1+(\ell+\epsilon)M\right]^{-1}\big[1-(\ell+\epsilon)M\big]>0.
\end{align*}
Denote by $\mathsf{B}_\eta$ the closed ball in $C_\xi([0,T];L^2(\Omega))$ centered at origin with radius $\eta$. Considering $\Phi:\mathsf{B}_\eta \to C_\xi([0,T];L^2(\Omega))$, we have
\begin{align*}
\|\Phi(u)(\cdot,t)\| & \le \omega(t,\lambda_1)\|\xi(\cdot,0)\| + \int_0^t \omega(t-s,\lambda_1)p(s)G(\|u[\xi]_\rho(\cdot,s)\|)ds,
\end{align*} 
thanks to Lemma \ref{lm-op-sol}(2). Put $\delta=\min\{\delta_0, \eta\}$. If $\xi\in C([-\tau,0];L^2(\Omega))$ such that $\|\xi\|_\infty\le \delta$, then 
$$
\|u[\xi]_\rho(\cdot,s)\|\le \|u\|_\infty+\|\xi\|_\infty \le \eta+\delta\le 2\eta \text{ for all } s\in [0,T].$$ 
So 
\begin{align*}
\|\Phi(u)(\cdot,t)\| & \le \omega(t,\lambda_1)\|\xi\|_\infty +  (\ell+\epsilon)\int_0^t \omega(t-s,\lambda_1)p(s)\|u[\xi]_\rho(\cdot,s)\|ds\\
& \le \omega(t,\lambda_1)\delta + (\eta+\delta) (\ell+\epsilon)\int_0^t \omega(t-s,\lambda_1)p(s)ds\\
& \le \big[\omega(t,\lambda_1)+(\ell+\epsilon)\omega(\cdot,\lambda_1)*p(t)\big]\delta_0+\eta(\ell+\epsilon)\omega(\cdot,\lambda_1)*p(t)\\
& \le \eta,\;\forall t\in [0,T].
\end{align*}
We have shown that $\Phi(\mathsf{B}_\eta)\subset \mathsf{B}_\eta$, provided $\|\xi\|_\infty \le \delta$. Consider $\Phi: \mathsf{B}_\eta\to \mathsf{B}_\eta$. In order to apply the Schauder fixed point theorem, it remains to check that $\Phi$ is a compact operator. It should be noted that, $\Phi$ admits the representation
$$
\Phi(u) = S(\cdot)\xi + \mathcal Q\circ N_f(u),
$$
where $N_f(u)(\cdot,t) = f(t,u[\xi]_\rho(\cdot,t))$. According to the compactness of $\mathcal Q$ stated in Lemma \ref{lm-cauchy-op}, we conclude that $\Phi$ is compact. The proof is complete.
\end{proof}
Theorem \ref{th-exist-a} deals with the case that $f$ is possibly superlinear. In the next theorem, we can relax the smallness condition on initial data, provided that $f$ has a sublinear growth.
\begin{theorem}\label{th-exist-b}
Let $f: [0,T]\times L^2(\Omega)\to L^2(\Omega)$ be a continuous mapping such that
\begin{itemize}
\item[(\textbf{F}2)] $\|f(t,v)\| \le p(t)(1+\|v\|)$ for all $t\in [0,T]$ and $v\in L^2(\Omega)$,
where $p\in L^1(0,T)$ is a nonnegative function. 
\end{itemize}
Then the problem \eqref{e1}-\eqref{e3} has at least one mild solution on $[-\tau,T]$.
\end{theorem}
\begin{proof}
Let $\psi\in C([0,T];\mathbb R)$ be the unique solution of the integral equation
\begin{align*}
\psi(t) = \|\xi\|_\infty + (1+\|\xi\|_\infty)\int_0^t p(s)ds + \int_0^t p(s)\psi(s)ds,
\end{align*}
and $D=\{u\in C_\xi([0,T];L^2(\Omega)): \sup_{\zeta\in [0,t]}\|u(\zeta)\|\le \psi(t),\; \forall t\in [0,T]\}$. Then $D$ is a closed and convex subset of $C_\xi([0,T];L^2(\Omega))$. Since $\Phi$ is continuous and compact, it suffices to show that $\Phi(D)\subset D$. Let $u\in D$, then
\begin{align*}
\|\Phi(u)(\cdot,t)\| & \le \omega(t,\lambda_1)\|\xi\|_\infty + \int_0^t \omega(t-s,\lambda_1)p(s)(1+\|u[\xi]_\rho(\cdot,s)\|)ds\\
& \le \|\xi\|_\infty + \int_0^t p(s)(1+ \|\xi\|_\infty+\sup_{\zeta\in [0,s]}\|u(\zeta)\| )ds.
\end{align*}
Since the last integral is nondecreasing in $t$, we get
\begin{align*}
\sup_{\zeta\in [0,t]}\|\Phi(u)(\cdot,\zeta)\| & \le \|\xi\|_\infty + \int_0^t p(s)(1+ \|\xi\|_\infty+\sup_{\zeta\in [0,s]}\|u(\zeta)\| )ds \\
&\le \|\xi\|_\infty + \int_0^t p(s)(1+ \|\xi\|_\infty+\psi(s))ds =  \psi(t),
\end{align*}
which ensures that $\Phi(u)\in D$. The proof is complete.
\end{proof}
In the next theorem, we state an existence and uniqueness result.
\begin{theorem}\label{th-exist-c}
Let $f: [0,T]\times L^2(\Omega)\to L^2(\Omega)$ be a continuous mapping such that
\begin{itemize}
\item[(\textbf{F}3)] $f(\cdot,0)=0$ and $\|f(t,v_1)-f(t,v_2)\| \le p(t)\kappa(r)\|v_1-v_2\|$ for all $t\in [0,T]$ and $v_1, v_2\in L^2(\Omega)$ such that $\|v_1\|,\|v_2\|\le r$,
where $p\in L^1(0,T)$ is a nonnegative function and $\kappa$ is a continuous function obeying that
\begin{align*}
\limsup_{r\to 0}\kappa(r) \cdot \sup_{t\in [0,T]}\int_0^t\omega(t-s,\lambda_1)p(s)ds <1.
\end{align*}
\end{itemize}
Then there exists $\delta>0$ such that the problem \eqref{e1}-\eqref{e3} has a unique mild solution on $[-\tau,T]$, provided $\|\xi\|_\infty\le \delta$.
\end{theorem}
\begin{proof}
The existence result can be obtained by applying Theorem \ref{th-exist-a} with $G(r) =\kappa(r)r$. It remains to prove the uniqueness. Assume that $u_1[\xi], u_2[\xi]$ are solutions of \eqref{e1}-\eqref{e3}. Let $R=\max\{\|u_1[\xi]\|_\infty, \|u_2[\xi]\|_\infty\}$, then 
\begin{align*}
\|u_1(\cdot,t)-u_2(\cdot,t)\| & \le \int_0^t p(s)\kappa(R)\|u_1[\xi]_\rho(\cdot,s)-u_2[\xi]_\rho(\cdot,s)\|ds\\
& \le \int_0^t p(s)\kappa(R)\sup_{\zeta\in [0,s]}\|u_1(\cdot,\zeta)-u_2(\cdot,\zeta)\|ds,
\end{align*}
due to the fact that $u_1(\cdot,s)=u_2(\cdot,s)$ for $s\in [-\tau,0]$. Observing that, the last integral is nondecreasing in $t$, we have
\begin{align*}
\sup_{\zeta\in [0,t]}\|u_1(\cdot,t)-u_2(\cdot,t)\| & \le \int_0^t p(s)\kappa(R)\sup_{\zeta\in [0,s]}\|u_1(\cdot,\zeta)-u_2(\cdot,\zeta)\|ds, \; t\in [0,T],
\end{align*}
which implies that $u_1=u_2$, by means of the Gronwall inequality. The proof is complete.
\end{proof}
We are now in a position to show the dissipativity of our system.
\begin{theorem}\label{th-stab-a}
Let the hypotheses of Theorem \ref{th-exist-b} hold for all $T>0$ and $\|p\|_\infty =\text{\rm esssup}_{t\ge 0}p(t)<\lambda_1$. Then there exists a bounded absorbing set for solution of \eqref{e1}-\eqref{e3} with arbitrary initial data.
\end{theorem}
\begin{proof}
Let $u$ be a solution of \eqref{e1}-\eqref{e3}. Then
\begin{align*}
\|u(\cdot,t)\| & \le \omega(t,\lambda_1)\|\xi\|_\infty+\int_0^t \omega(t-s,\lambda_1)p(s)(1+\|u_\rho(\cdot,s)\|)ds\\
& \le \omega(t,\lambda_1)\|\xi\|_\infty+\int_0^t \omega(t-s,\lambda_1)\|p\|_\infty(1+\sup_{\zeta\in [s-\rho(s),s]}\|u(\cdot,\zeta)\|)ds.
\end{align*}
Applying the Halanay type inequality formulated in Lemma \ref{lm-halanay}, we get
\begin{align*}
\limsup_{t\to \infty}\|u(\cdot,t)\| & \le \sup_{t\in\mathbb R^+}\int_0^t\omega(t-s,\lambda_1)\|p\|_\infty ds\\
& = \|p\|_\infty\lambda_1^{-1} \sup_{t\in\mathbb R^+}(1-\omega(t,\lambda_1))=\|p\|_\infty\lambda_1^{-1}.
\end{align*}
This implies that the ball $B(0,R)\subset L^2(\Omega)$ with $R=\|p\|_\infty\lambda_1^{-1}+1$ turns out to be an absorbing set for solution of \eqref{e1}-\eqref{e3} with arbitrary initial data.
\end{proof}
The next theorem shows the asymptotic stability of zero solution to \eqref{e1}.
\begin{theorem}\label{th-stab-b}
Let $f: \mathbb R^+\times L^2(\Omega)\to L^2(\Omega)$ be a continuous mapping such that
\begin{itemize}
\item[(\textbf{F}4)] $f(\cdot,0)=0$ and $\|f(t,v_1)-f(t,v_2)\| \le p(t)\kappa(r)\|v_1-v_2\|$ for all $t\in \mathbb R^+$ and $v_1, v_2\in L^2(\Omega)$ such that $\|v_1\|,\|v_2\|\le r$,
where $p\in L^\infty(\mathbb R^+)$ is a nonnegative function and $\kappa$ is a continuous function satisfying that
\begin{align*}
\|p\|_\infty \cdot \limsup_{r\to 0}\kappa(r) <\lambda_1.
\end{align*}
\end{itemize}
Then the zero solution of \eqref{e1} is asymptotically stable.
\end{theorem}
\begin{proof}
Let $\ell=\limsup\limits_{r\to 0}\kappa(r)$. Choosing $\theta>0$ such that $\|p\|_\infty(\ell+\theta)<\lambda_1$, we can find $\eta>0$ such that $\kappa(r)\le \ell+\theta$ for all $r\in [0,2\eta]$. Reasoning as in the proof of Theorem \ref{th-exist-a} and \ref{th-exist-c}, there exists $\delta>0$ such that the problem \eqref{e1}-\eqref{e3} has a unique mild solution $u\in\mathsf{B}_\eta$ as long as $\|\xi\|_\infty\le \delta$, which is defined on $[-\tau, T]$ for all $T>0$. Moreover, one has the following estimate
\begin{align*}
\|u(\cdot,t)\| & \le \omega(t,\lambda_1)\|\xi\|_\infty + \int_0^t \omega(t-s,\lambda_1)p(s)(\ell+\theta)\|u[\xi]_\rho(\cdot,s)\|ds\\
& \le \omega(t,\lambda_1)\|\xi\|_\infty + \int_0^t \omega(t-s,\lambda_1)\|p\|_\infty(\ell+\theta)\sup_{\zeta\in [s-\rho(s),s]}\|u(\cdot,\zeta)\|ds.
\end{align*}
Employing Lemma \ref{lm-halanay} with $b(\cdot)=0, a=\|p\|_\infty(\ell+\theta)$, we obtain
\begin{align*}
\|u(\cdot,t)\| & \le \left(\frac{\lambda_1}{\lambda_1-\|p\|_\infty(\ell+\theta)}+1\right)\|\xi\|_\infty,\;\forall t\ge 0,\\
\lim_{t\to \infty}\|u(\cdot,t)\| & = 0,
\end{align*}
which imply the asymptotic stability of the zero solution of \eqref{e1}. The proof is complete.
\end{proof}
\section{Existence of decay solutions}
Our goal of this section is to prove the existence of decay solutions to the problem \eqref{e1}-\eqref{e3} under the assumption that, the nonlinearity is non-Lipschitzian and possibly superlinear. Specifically, assume that
\begin{itemize}
\item[(\textbf{F}5)] $f: \mathbb R^+\times L^2(\Omega))\to L^2(\Omega)$ is a continuous mapping such that 
\begin{align*}
\|f(t,v)\| \le p(t)G(\|v\|),\;\forall t\in\mathbb R^+, v\in L^2(\Omega),
\end{align*}
where $p\in L^1_{loc}(\mathbb R^+)$ is a nonnegative function and $G\in C(\mathbb R^+)$ is a nonnegative and  nondecreasing function such that
\begin{align}
\limsup\limits_{r\to 0} \frac{G(r)} r \cdot \sup_{t\ge 0}\int_0^t \omega(t-\tau,\lambda_1)p(\tau)d\tau <1, \label{F3a}
\end{align}
and
\begin{align}
\lim_{T\to \infty}\sup_{t\ge T}\int_0^{\frac t2}\omega(t-\tau,\lambda_1)p(\tau)d\tau=0.\label{F3b}
\end{align}
\end{itemize}
In order to study the existence of decay solutions to \eqref{e1}-\eqref{e3}, we make use of the fixed point theory for condensing maps.
\begin{definition}\cite{KOZ}\label{def-mnc}
Let $E$ be a Banach space and $\mathcal P_b(E)$ the collection of all nonempty and bounded subsets of $E$. A function $\mu: \mathcal P_b(E)\to \mathbb R^+$ is said to be a measure of noncompactness (MNC) if $\mu(\overline{\text{\rm co}}\,D)=\mu(D)$ for all $D\in\mathcal P_b(E)$, here the notation $\overline{\text{\rm co}}$ denote the closure of convex hull of subsets in $E$. An MNC is called
\begin{itemize}
\item nonsingular if $\mu(D\cup\{x\}) = \mu(D)$ for all $D\in\mathcal P_b(E)$, $x\in E$.
\item monotone if $\mu(D_1)\le \mu(D_2)$ provided that $D_1\subset D_2$.
\end{itemize}
\end{definition}
The MNC defined by 
\begin{align*}
\chi(D) = \inf\{\varepsilon>0: D\text{ admits a finite }\varepsilon-\text{net}\}
\end{align*}
is called the Hausdorff measure of noncompactness.
\begin{definition}\cite{KOZ}
Let $E$ be a Banach space and $D\in \mathcal P_b(E)$. A continuous map $\mathcal F: D\to E$ is said to be condensing with respect to MNC $\mu$ ($\mu$-condensing) iff the relation $\mu(B)\le \mu(\mathcal F(B)), B\subset D$, implies that $B$ is relatively compact. 
\end{definition}
The following theorem states a fixed point principle for condensing maps.
\begin{theorem}\cite{KOZ}\label{th-fix}
Let $\mu$ be a monotone and nonsingular MNC on $E$. Assume that $D\subset E$ is a closed convex set and $\mathcal F:D\to D$ is $\mu$-condensing. Then $\mathcal F$ admits a fixed point.
\end{theorem}
Let $BC_0(\mathbb R^+;L^2(\Omega))$ be the space of continuous functions on $\mathbb R^+$, taking values in $L^2(\Omega)$ and decaying as $t\to \infty$. Given $\xi\in C([-\tau,0];L^2(\Omega))$, put $\mathcal{BC}_0^\xi=\{u\in BC_0(\mathbb R^+;L^2(\Omega)): u(\cdot, 0)=\xi(\cdot,0)\}$. Then $\mathcal{BC}_0^\xi$ with the supremum norm $\|\cdot\|_\infty$ is a closed subspace of $BC_0(\mathbb R^+;L^2(\Omega))$.

Let $D$ be a bounded set in $\mathcal{BC}_0^\xi$ and $\pi_T: \mathcal{BC}_0^\xi\to C([0,T];L^2(\Omega))$ the restriction operator on $\mathcal{BC}_0^\xi$, i.e. $\pi_T(u)$ is the restriction of $u\in \mathcal{BC}_0^\xi$ to the interval $[0,T]$. Define 
\begin{align*}
d_\infty(D) & = \lim_{T\to\infty}\sup_{u\in D}\sup_{t\ge T}\|u(\cdot,t)\|,\\
\chi_\infty(D) & = \sup_{T>0}\chi_T(\pi_T(D)),
\end{align*}
where $\chi_T(\cdot)$ is the Hausdorff MNC in $C([0,T];L^2(\Omega))$. Then the following MNC defined in \cite{AK},
\begin{equation}\label{mnc-star}
\chi^*(D) = d_\infty(D) + \chi_\infty(D),
\end{equation}
possesses all properties stated in Definition \ref{def-mnc}. In addition, if $\chi^*(D)=0$ then $D$ is relatively compact in $BC_0(\mathbb R^+;L^2(\Omega))$. 
\begin{lemma}\label{lm-radius}
Let (\textbf{F}5) hold. Then there exist positive numbers $\delta$ and $\eta$ such that for $\|\xi\|_\infty\le \delta$, the solution operator $\Phi$ obeys $\Phi(\mathsf{B}_\eta)\subset \mathsf{B}_\eta$, where $\mathsf{B}_\eta$ is the closed ball in $\mathcal{BC}_0^\xi$ centered at origin with radius $\eta$.
\end{lemma}
\begin{proof}
Denote 
\begin{align*}
\ell = \limsup_{r\to 0}\frac{G(r)}{r},\; M=\sup_{t\ge 0}\int_0^t \omega(t-\tau,\lambda_1)p(\tau)d\tau.
\end{align*} 
Then by \eqref{F3a}, one can take $\zeta>0$ such that
\begin{equation}\label{lm-radius0}
(\ell+\zeta)M<1.
\end{equation}
Moreover, there exists $\eta>0$ such that $\frac{G(r)}{r}\le \ell+\zeta$ for all $r\in (0,2\eta]$. Recall that the solution operator $\Phi$ is defined by
\begin{align*}
\Phi(u)(\cdot, t) = S(t)\xi(\cdot,0) + \int_0^t S(t-s)f(s,u[\xi]_\rho(\cdot,s))ds,\; u\in \mathcal{BC}_0^\xi.
\end{align*}
Considering the operator $\Phi$ on $\mathsf{B}_\eta$ with $\|\xi\|_\infty\le \eta$, we have
\begin{align}
\|\Phi(u)(\cdot, t)\| & \le \omega(t,\lambda_1)\|\xi\|_\infty + \int_0^t \omega(t-\tau,\lambda_1)p(\tau)G(\|u[\xi]_\rho(\cdot,s)\|)ds.\label{lm-radius1}
\end{align}
We first check that $\Phi(u)\in \mathcal{BC}_0^\xi$, provided $u\in \mathcal{BC}_0^\xi$. It suffices to prove that $\Phi(u)(\cdot, t)\to 0$ as $t\to \infty$ in $L^2(\Omega)$. According to \eqref{lm-radius1}, one has to testify that
$$
I(t):= \int_0^t \omega(t-s,\lambda_1)p(s)G(\|u[\xi]_\rho(\cdot,s)\|)ds\to 0\text{ as } t\to \infty.
$$
Since $t-\rho(t)\to\infty$ as $t\to \infty$, we get $\|u[\xi](\cdot, t-\rho(t))\|\to 0$ as $t\to \infty$. So for any $\varepsilon>0$, there exists $T>0$ such that $G(\|u[\xi](\cdot,s-\rho(s))\|)\le \varepsilon$ for all $s\ge T$, thanks to the fact that $G$ is continuous and $G(0)=0$. Hence for $t > T$, we get
\begin{align*}
I(t) & =\left( \int_0^T + \int_T^t\right)\omega(t-s,\lambda_1)p(s)G(\|u[\xi](\cdot,s-\rho(s))\|)d\tau\\
& \le G(2\eta)\int_0^T \omega(t-s,\lambda_1)p(s)ds + \varepsilon \int_T^t \omega(t-s,\lambda_1)p(s)ds\\
& \le G(2\eta)\omega(t-T,\lambda_1)\int_0^T p(s)ds + \varepsilon M\\
& \le [G(2\eta)+M]\varepsilon,
\end{align*}
for all $t$ chosen so that
$$
\omega(t-T,\lambda_1)\int_0^T p(s)ds <\varepsilon,
$$
which is possible since $\omega(t,\lambda_1)\to 0$ as $t\to\infty$. We have proved that $\Phi(u)\in\mathcal{BC}_0^\xi$. Let
\begin{equation}\label{lm-radius2}
\delta_0=\eta \inf_{t\ge 0}\left[\big(\omega(t,\lambda_1)+(\ell+\zeta)\omega(\cdot,\lambda_1)*p(t)\big)^{-1}\big(1-(\ell+\zeta)\omega(\cdot,\lambda_1)*p(t)\big)\right],
\end{equation}
then $\delta_0>0$. Indeed, one has
\begin{align*}
\big(\omega(t,\lambda_1)+(\ell+\zeta)\omega(\cdot,\lambda_1)*p(t)\big)^{-1}& \ge [1+(\ell+\zeta)M]^{-1},\; \forall t\ge 0,
\end{align*}
then
\begin{align*}
\delta_0 & \ge \eta [1+(\ell+\zeta)M]^{-1}\inf_{t\ge 0}\left(1-(\ell+\zeta)\int_0^t \omega(t-\tau,\lambda_1)p(\tau)d\tau\right)\\
& \ge \eta [1+(\ell+\zeta)M]^{-1}\left(1-(\ell+\zeta)\sup_{t\ge 0}\int_0^t \omega(t-\tau,\lambda_1)p(\tau)d\tau\right)>0,
\end{align*}
thanks to \eqref{lm-radius0}. Choosing $\delta = \min\{\eta,\delta_0\}$, we show that $\Phi(u)\in\mathsf{B}_\eta$ provided $\|\xi\|_\infty\le \delta$. For $\|\xi\|\le \delta$, $u\in\mathsf{B}_\eta$, we get $\|u[\xi](\cdot,s)\|\le 2\eta$ for any $s\ge -\tau$. In addition,
\begin{align*}
\|\Phi(u)(\cdot,t)\| & \le \omega(t,\lambda_1)\|\xi\|_\infty + \int_0^t \omega(t-s,\lambda_1)p(s)G(\|u[\xi](\cdot,s-\rho(s))\|)ds\\
& \le \omega(t,\lambda_1)\|\xi\|_\infty + \int_0^t \omega(t-s,\lambda_1)p(s)(\ell+\zeta)\|u[\xi](\cdot,s-\rho(s))\|ds\\
& \le \omega(t,\lambda_1)\|\xi\|_\infty + \int_0^t \omega(t-s,\lambda_1)p(s)(\ell+\zeta)(\|u\|_\infty+\|\xi\|_\infty)ds\\
& \le \omega(t,\lambda_1)\|\xi\|_\infty + (\eta+\|\xi\|_\infty)(\ell+\zeta) \int_0^t \omega(t-s,\lambda_1)p(s)ds\\
& \le  \big[\omega(t,\lambda_1) + (\ell+\zeta)\omega(\cdot,\lambda_1)*p(t)\big] \|\xi\|_\infty+ \eta(\ell+\zeta)\omega(\cdot,\lambda_1)*p(t)\\
& \le  \big[\omega(t,\lambda_1) + (\ell+\zeta)\omega(\cdot,\lambda_1)*p(t)\big] \delta_0 + \eta(\ell+\zeta)\omega(\cdot,\lambda_1)*p(t)\\
& \le \eta,\;\forall t\ge 0,
\end{align*}
due to the formulation of $\delta_0$ in \eqref{lm-radius2}. Therefore $\Phi(\mathsf{B}_\eta)\subset \mathsf{B}_\eta$. The proof is complete.
\end{proof}
The following theorem represents the main result of this section.
\begin{theorem}\label{th-decay}
Let (\textbf{F}5) hold. Then there exists $\delta>0$ such that, the problem \eqref{e1}-\eqref{e3} has a compact set of decay solutions, provided $\|\xi\|_\infty \le \delta$.
\end{theorem}
\begin{proof}
Taking $\delta$ and $\mathsf{B}_\eta$ from Lemma \ref{lm-radius}, we consider the solution map $\Phi: \mathsf{B}_\eta\to \mathsf{B}_\eta$. By standard reasoning, we get that $\Phi$ is continuous. We will show that $\Phi$ is $\chi^*$-condensing. 
Let $D\subset \mathsf{B}_\rho$. Then arguing as in the proof of Theorem \ref{th-exist-a}, one has $\pi_T\circ \Phi$ is a compact mapping, i.e., $\pi_T(\Phi(D))$ is relatively compact in $C([0,T];L^2(\Omega))$. This implies $\chi_T(\pi_T(\Phi(D)))=0$ and then
$\chi_\infty(\Phi(D))=0$. We are now in a position to estimate $d_\infty(\Phi(D))$.

Let $z\in \Phi(D)$ and $u\in D$ be such that $z=\Phi (u)$. Then
\begin{align*}
\|z(\cdot,t)\| & \le \omega(t,\lambda_1)\|\xi\|_\infty + \int_0^t \omega(t-s,\lambda_1)p(s)G(\|u[\xi](\cdot,s-\rho(s))\|)ds\\
& \le \omega(t,\lambda_1)\|\xi\|_\infty + (\ell + \zeta)\int_0^t \omega(t-s,\lambda_1)p(s)\|u[\xi](\cdot,s-\rho(s))\|ds\\
& \le \omega(t,\lambda_1)\|\xi\|_\infty + (\ell + \zeta)\left(\int_0^{\frac t2}+\int_{\frac t2}^t \right)\omega(t-s,\lambda_1)p(s)\|u[\xi](\cdot,s-\rho(s))\|ds\\
& \le \omega(t,\lambda_1)\|\xi\|_\infty + 2(\ell + \zeta)\eta \int_0^{\frac t2}\omega(t-s,\lambda_1)p(s)ds\\
& \quad + \sup_{s \ge \frac t2}\|u[\xi](\cdot,s-\rho(s))\|(\ell + \zeta)\int_{\frac t2}^t \omega(t-s,\lambda_1)p(s)ds.
\end{align*}
Noting that, for given $T>0$, one can find $T_1>T$ such that $t-\rho(t)\ge T$ for all $t\ge T_1$.  So for $t\ge 2 T_1$, we have
\begin{align*}
\|z(\cdot,t)\| & \le \omega(t,\lambda_1)\|\xi\|_\infty + 2(\ell + \zeta)\eta \int_0^{\frac t2}\omega(t-s,\lambda_1)p(s)ds\\
& \quad +\sup_{s \ge T}\|u(\cdot,s)\|(\ell + \zeta)\int_0^t \omega(t-s,\lambda_1)p(s)ds\\
& \le \omega(t,\lambda_1)\|\xi\|_\infty + 2(\ell + \zeta)\eta \int_0^{\frac t2}\omega(t-s,\lambda_1)p(s)ds\\
& \quad +\sup_{u\in D}\sup_{s \ge T}\|u(\cdot,s)\|(\ell + \zeta)\int_0^t \omega(t-s,\lambda_1)p(s)ds.
\end{align*}
Then it follows that
\begin{align*}
\sup_{t\ge 2T_1}\|z(\cdot,t)\| & \le \omega(2T_1,\lambda_1)\|\xi\|_\infty + 2(\ell + \zeta)\eta \sup_{t\ge 2T_1}\int_0^{\frac t2}\omega(t-s,\lambda_1)p(s)ds\\
& \quad + \sup_{u\in D}\sup_{s \ge T}\|u(\cdot,s)\|(\ell + \zeta)M,
\end{align*}
where 
$$
M=\sup_{t\ge 0}\int_0^t \omega(t-s,\lambda_1)p(s)ds.
$$
Since $z\in \Phi(D)$ is taken arbitrarily, we get
\begin{align*}
\sup_{z\in\Phi(D)}\sup_{t\ge 2T_1}\|z(\cdot,t)\| & \le \omega(2T_1,\lambda_1)\|\xi\|_\infty + 2(\ell + \zeta)\eta \sup_{t\ge 2T_1}\int_0^{\frac t2}\omega(t-s,\lambda_1)p(s)ds\\
& \quad + \sup_{u\in D}\sup_{s \ge T}\|u(\cdot,s)\|(\ell + \zeta)M,
\end{align*}
which ensures that
\begin{align*}
d_\infty(\Phi(D))\le (\ell + \zeta)Md_\infty(D),
\end{align*}
thanks to \eqref{F3b} and the fact that $T_1\to \infty$ as $T\to \infty$. Therefore,
\begin{align*}
\chi^*(\Phi(D)) & = \chi_\infty(\Phi(D)) + d_\infty(\Phi(D)) = d_\infty(\Phi(D)) \le (\ell + \zeta)Md_\infty(D)\\
& \le (\ell + \zeta)M [d_\infty(D)+\chi_\infty(D)] = (\ell + \zeta)M\chi^*(D).
\end{align*}
Now if $\chi^*(D)\le \chi^*(\Phi(D))$ then $\chi^*(D)\le (\ell + \zeta)M\chi^*(D)$ which implies $\chi^*(D)=0$, thanks to the fact that $(\ell + \zeta)M<1$. Thus $\Phi$ is $\chi^*$-condensing and it admits a fixed point, according to Theorem \ref{th-fix}. Denote by $\mathcal D$ the fixed point set of $\Phi$ in $\mathsf{B}_\eta$. Then $\mathcal D$ is closed and $\mathcal D\subset \Phi(\mathcal D)$. Hence,
$$
\chi^*(\mathcal D)\le \chi^*(\Phi(\mathcal D))\le (\ell + \zeta)M\chi^*(\mathcal D),
$$
which ensures $\chi^*(\mathcal D)=0$ and $\mathcal D$ is a compact set. The proof is complete.
\end{proof}


\begin{thebibliography}{00}
\bibitem{AK} N.T. Anh, T.D. Ke, \textit{Decay integral solutions for neutral fractional differential equations with infinite delays}, Math. Methods Appl. Sci. 38 (2015), 1601-1622.

\bibitem{AKQ}  N.T. Anh, T.D. Ke, N.N. Quan, \textit{Weak stability for integro-differential inclusions of diffusion-wave type involving infinite delays}, Discrete Contin. Dyn. Syst. Ser. B 21 (2016), 3637-3654.

\bibitem{Bazh15}  E. Bazhlekova, B. Jin, R. Lazarov, Z. Zhou, \textit{An analysis of the Rayleigh-Stokes problem for a generalized second-grade fluid}, {Numer. Math.} 131 (2015), no. 1, 1-31.

\bibitem{Bi18}   X. Bi, S. Mu, Q. Liu, Q. Liu, B. Liu, P. Zhuang, J. Gao, H. Jiang, X. Li, B. Li, \textit{Advanced implicit meshless approaches for the Rayleigh-Stokes problem for a heated generalized second grade fluid with fractional derivative}, {Int. J. Comput. Methods} 15 (2018), no. 5, 1850032, 27 pp.

\bibitem{Chen13}   C.M. Chen, F. Liu, K. Burrage, Y. Chen, \textit{Numerical methods of the variable-order Rayleigh-Stokes problem for a heated generalized second grade fluid with fractional derivative}, {IMA J. Appl. Math.} 78 (2013), no. 5, 924-944. 

\bibitem{Chen08}  C.M. Chen, F. Liu, V. Anh, \textit{Numerical analysis of the Rayleigh-Stokes problem for a heated generalized second grade fluid with fractional derivatives}, {Appl. Math. Comput.} 204 (2008), no. 1, 340-351.

\bibitem{Drabek}  P. Dr\'abek, J. Milota, Methods of nonlinear analysis. Applications to differential equations. Birkh\"auser Verlag, Basel, 2007. 

\bibitem{Evans} L.C. Evans, Partial differential equations. Second edition. American Mathematical Society, Providence, RI, 2010.

\bibitem{FJFV09}  C. Fetecau, M. Jamil, C. Fetecau, D. Vieru, \textit{The Rayleigh-Stokes problem for an edge in a generalized Oldroyd-B fluid}, {Z. Angew. Math. Phys.} 60 (2009), no. 5, 921-933.

\bibitem{KOZ} M. Kamenskii, V. Obukhovskii, P. Zecca, Condensing multivalued maps and semilinear differential inclusions in Banach spaces, Walter de Gruyter, Berlin, New York, 2001.

\bibitem{KL}  T.D. Ke, D. Lan, \textit{Fixed point approach for weakly asymptotic stability of fractional differential inclusions involving impulsive effects}, J. Fixed Point Theory Appl. 19 (2017), 2185-2208.

\bibitem{Khan09}  M. Khan, \textit{The Rayleigh-Stokes problem for an edge in a viscoelastic fluid with a fractional derivative model}, {Nonlinear Anal. Real World Appl.} 10 (2009), no. 5, 3190-3195.

\bibitem{Lan21} D. Lan, \textit{Regularity and stability analysis for semilinear generalized Rayleigh-Stokes equations}, Evol. Equ. Control Theory 2021, doi: 10.3934/eect.2021002

\bibitem{Luc19}  N.H. Luc, N.H. Tuan, Y. Zhou, \textit{Regularity of the solution for a final value problem for the Rayleigh-Stokes equation}, {Math. Methods Appl. Sci.} 42 (2019), no. 10, 3481-3495.

\bibitem{Ngoc21} T. B. Ngoc, N. H. Luc, V. V. Au, N. H. Tuan, Y. Zhou, \textit{Existence and regularity of inverse problem for the nonlinear fractional Rayleigh-Stokes equations}, Math. Methods Appl. Sci. 43 (2021), 2532-2558.

\bibitem{Pruss} J. Pr\"uss, Evolutionary Integral Equations and Applications. Monographs in Mathematics 87,
Birkh\"auser, Basel, 1993.

\bibitem{SSM18}  F. Salehi, H. Saeedi, M.M. Moghadam, \textit{Discrete Hahn polynomials for numerical solution of two-dimensional variable-order fractional Rayleigh-Stokes problem}, {Comput. Appl. Math.} 37 (2018), no. 4, 5274-5292.

\bibitem{STZM06} F. Shen, W. Tan, Y. Zhao, Y. Masuoka, \textit{The Rayleigh-Stokes problem for a heated generalized second grade fluid with fractional derivative model}, {Nonlinear Anal. Real World Appl.} 7 (2006), no. 5, 1072-1080.

\bibitem{Tuan19}  N.H. Tuan, Y. Zhou, T.N. Thach, N.H. Can, \textit{Initial inverse problem for the nonlinear fractional Rayleigh-Stokes equation with random discrete data}, {Commun. Nonlinear Sci. Numer. Simul.} 78 (2019), 104873, 18 pp.

\bibitem{Wang15}  D.Wang,  A. Xiao, H. Liu, \textit{Dissipativity and stability analysis for fractional functional differential equations}, Fract. Calc. Appl. Anal. 18 (2015), no. 6, 1399-1422.

\bibitem{XN09}  C. Xue, J. Nie, \textit{Exact solutions of the Rayleigh-Stokes problem for a heated generalized second grade fluid in a porous half-space}, {Appl. Math. Model.} 33 (2009), no. 1, 524-531.

\bibitem{Zaky18}  M.A. Zaky,  \textit{An improved tau method for the multi-dimensional fractional Rayleigh-Stokes problem for a heated generalized second grade fluid}, {Comput. Math. Appl.} 75 (2018), no. 7, 2243-2258.

\bibitem{ZBF07}  J. Zierep, R. Bohning, C. Fetecau, \textit{Rayleigh-Stokes problem for non-Newtonian medium with memory}, {ZAMM Z. Angew. Math. Mech.} 87 (2007), no. 6, 462-467.

\end{thebibliography}
\end{document}